\documentclass{amsart}
\usepackage{amsmath,amssymb,graphicx}

\newcounter{sub}
\newenvironment{sub}%
{\begin{list}{(\arabic{sub})}{\usecounter{sub}%
\setlength{\leftmargin}{2em}}}{\end{list}}

\def\ch{\mbox{\rm char }}

\makeatletter
    
    \@addtoreset{equation}{section}
\makeatother

\newtheorem{thm}{Theorem}[section]
\newtheorem{cor}{Corollary}[section]
\newtheorem{prop}{Proposition}[section]
\newtheorem{rem}{Remark}[section]
\newtheorem{exmp}{Example}[section]

\title{Rationality problem of generalized Ch\^atelet surfaces}
\author{Aiichi Yamasaki}

\begin{document}

\maketitle

\begin{abstract}
The surface $z^2=ay^2+P(x), \,
a \in k, \, P(x) \in k[x]$ is not $k$-rational,
if $a \not\in k^2$ and $P(x)$ satisfies some conditions.
This result essentially due to Iskovskih
\cite{Isk67}, \cite{Isk96}, \cite{Isk97},
but his statement is in terms of algebraic geometry,
and not so easy to access for the researchers
of the field extension.
This paper aims to give a formulation accessible more easily.
A necessary and sufficient condition for $k$-rationality
is given in terms of $a$ and $P$,
assuming that $\ch k \not= 2$
and every irreducible component of $P(x)$ is separable over $k$.
\end{abstract}

\section{Introduction}

Iskovskih studied the rationality of conic bundles
and obtained the following result \cite{Isk67}, \cite{Isk96}, \cite{Isk97}.

\begin{thm} (Iskovskih)
Let $X$ be a rational $k$-surface fibrated as a standard conic bundle
$\pi: X \rightarrow \mathbb{P}_k^1$.
If $X$ has at least four degenerate geometric fibres,
then $X$ is not $k$-rational.
\end{thm}

The function field of a conic bundle may be a quadratic extension of $k(x,y)$,
the two dimensional rational function field over $k$, defined as
$$z^2=Q(x)y^2+P(x), \quad P,Q \in k(x).$$
So Iskovskih's problem is equivalent with the rationality problem of the above
$k(x,y,z)$.
However it is not clear how the conditions in Iskovskih's theorem
can be restated in terms of $P$ and $Q$.

We need a necessary and suficient condition for $k$-rationality
in terms of $P$ and $Q$.
As a first step,
we shall consider the case $\deg Q=0$,
namely we shall consider the rationality problem of
\begin{equation} \label{cha}
z^2=ay^2+P(x), \quad a \in k, \, P(x) \in k(x),
\end{equation}
and find the condition in $P(x)$ for $k$-rationality.

\begin{rem}
The surface defined by (\ref{cha}) is called a Ch\^atelet surface
when $\deg P=3$ or $4$,
which was first studied by Ch\^atelet \cite{Cha59}.
The surface may be called a generalized Ch\^atelet surface
when $\deg P \geq 5$.
\end{rem}

\subsection{Main result.} \label{main}
In this paper, we prove that $k(x,y,z)$ defined by (\ref{cha})
is not $k$-rational under the following (1) $\sim$ (5).
\begin{sub}
\item[(1)]
$a \not\in k^2$.
\item[(2)]
$P(x) \in k[x]$, $\deg P \geq 3$
and the irreducible decomposition of $P$ is square-free,
namely $P(x)$ is a product of mutually different irreducible polynomials.
\item[(3)]
Let $l$ be the splitting field of $P(x)$,
then $k(\sqrt{a}) \subset l$.
\item[(4)]
Every irreducible component of $P(x)$ is irreducible
also over $k(\sqrt{a})$.
\item[(5)]
$\ch k \not=2$,
and every irreducible component of $P(x)$ is reparable over $k$.
\end{sub}
Among these conditions, only (5) is an assumption,
and other (1) $\sim$ (4) are merely exclusions of known cases,
as explained below.

\begin{sub}
\item[(1)]
If $a \in k^2$, then $k(x,y,z)$ is $k$-rational.
$a \in k^2$is equivalent to $\sqrt{a} \in k$,
so putting $u=z+\sqrt{a}y$ and $v=z-\sqrt{a}y$,
(\ref{cha}) becomes $uv=P(x)$, thus $k(x,y,z)=k(x,u,v)=k(x,u)$
since $v=\frac{P(x)}{u} \in k(x,u)$.
\item[(2)]
Evidently the rationality of (\ref{cha}) depends on $k^2$-coset of $a$
and $k(x)^2$-coset of $P(x)$.
So without loss of generality,
we can assume that $P(x) \in k[x]$
and the irreducible decomposition of $P$ is square-free.

When $\deg P=1$,
(\ref{cha}) is written as $z^2=ay^2+x$
so $k(x,y,z)=k(y,z)$ is $k$-rational..

When $\deg P=0$ or $2$, the problem is reduced to a problem of quadratic forms,
and we have the following result.

When $\deg P=0$,
(\ref{cha}) is written as $z^2=ay^2+b$,
and $k(x,y,z)$ is $k$-rational if and only if $b \in k^2-ak^2$.

When $\deg P=2$ and $\ch k \not= 2$,
(\ref{cha}) is written as $z^2=ay^2+bx^2+c$,
and if $c \not=0$, $k(x,y,z)$ is $k$-rational if and only if $c \in k^2-ak^2-bk^2$.
(If $c=0$, $k(x,y,z)$ is $k$-rational if and only if $b \in k^2-ak^2$).

\item[(3)]
If $l \cap k(\sqrt{a})=k$,
$k(x,y,z)$ is not $k$-rational
as Manin proved in \cite{Man67},
using the criterion I explained later in this section.

\item[(4)]
Suppose that some irreducible component $P_1$ of $P$ is
reducible over $k(\sqrt{a})$, and one of the irreducible
component of $P_1$ over $k(\sqrt{a})$ is $A(x)+\sqrt{a}B(x)$ with
$A(x), B(x) \in k[x]$.
Then its conjugate is $A(x)-\sqrt{a}B(x)$,
so we have $P_1(x)=A(x)^2-aB(x)^2$.
 Putting $z=A(x)z^\prime+aB(x)y^\prime$ and $y=B(x)z^\prime+A(x)y^\prime$,
 we have $z^2-ay^2=P_1(x)(z^{\prime^2}-ay^{\prime2})$,
 so that $z^{\prime2}-ay^{\prime2}=P(x)/P_1(x)$.
 Since $k(x,y,z)=k(x,y^\prime,z^\prime)$,
 the rationality of $k(x,y,z)$ does not change by replacing $P$ with $P/P_1$.
In other words, omission of $P_1$ from the irreducible decomposition of $P$
does not change the rationality problem.
\end{sub}

\subsection{Idea of the proof}

We shall explain the details in later sections,
and here state only the outline.

Let $l$ be the splitting field of $P(x)$,
then $\sqrt{a} \in l$ because of the condition 3,
so $l(x,y,z)$ is $l$-rational and $l(x,y,z)=l(x,u)$
where $u=z+\sqrt{a}y$.
$L$ is a Galois extension of $k$.
Put $\mathfrak{G}=Gal(l/k)$ and $N=Gal(l/k\big(\sqrt{a})\big)$.
$\mathfrak{G}$ acts on $x$ trivially,
and $N$ acts on $u$ trivially,
while $\sigma \in \mathfrak{G} \setminus N$ acts as
$u \mapsto z-\sqrt{a}y=\frac{P(x)}{u}$.

The automorphism $T:(x,u) \mapsto (x,\frac{P(x)}{u})$ of $k(x,u)$
induces a $\overline{k}$-birational transformation of $\mathbb{P}^1 \times \mathbb{P}^1$.
After successive blowings-up at fundamental points of $T$,
it becomes $\overline{k}$-biregular on the obtained surface $X$.

Then $\mathfrak{G}$ acts on $X$ homeomorphically in Zariski topology,
and induces a permutation of irreducible curves.
Thus the divisor group $Div(X)$ becomes a permutation $\mathfrak{G}$-module.
Since the principal divisor group is kept invariant by the action of $\mathfrak{G}$,
Picard group $Pic(X)$ is also a $\mathfrak{G}$-module.

From the structure of $Pic(X)$ as a $\mathfrak{G}$-module,
we can derive the $k$-irrationality of $k(x,y,z)$.
We use three criteria I, II, III stated below.
The proof is overlapping, and actually only II and III are sufficient
for the proof of the main result,
but we shall explain all of I, II, III.

\noindent
I. The first Galois cohomology $H^1(\mathfrak{G},Pic(X))$ is
$k$-birational invariant, as Manin pointed out in \cite{Man67}.
Especially, if $k(x,y,z)$ is $k$-rational
(namely if $X$ is $k$-birational with $\mathbb{P}^1 \times \mathbb{P}^1$),
we should have $H^1(\mathfrak{G},Pic(X))=0$.
In other words,
$H^1(\mathfrak{G},Pic(X)) \not=0$ is a criterion for the $k$-irrationality of $k(x,y,z)$.

On the other hand, we can prove
\begin{thm} \label{A}
$H^1(\mathfrak{G},Pic(X))=(\mathbb{Z}/2\mathbb{Z})^j$
where $j$ is as follows.
Let $r^\prime$ be the number of irreducible components of $P(x)$.
When $\deg P$ is odd, then $j=r^\prime-1$.
When $\deg P$ is even,
if every irreducible component is of even degree,
then $j=r^\prime-1$,
while if some irreducible component is of odd degree,
then $j=r^\prime-2$.
\end{thm}

Theorem \ref{A} implies that $k(x,y,z)$ is not $k$-rational
except when $P(x)$ is irreducible,
or a product of two irreducible polynomials of odd degree.

\noindent
II. We shall reach a contradiction by the calculation of the intersection form.

If $X$ is $k$-birational with $\mathbb{P}^1 \times \mathbb{P}^1$,
there exist two families of $\mathfrak{G}$-invariant irreducible curves
$\{C_a\}$ and $\{C_a^\prime\}$ on $X$, parametrized by elements of $k$.

After successive blowings-up at fundamental points of $\mathbb{P}^1 \times \mathbb{P}^1$
and $X$ respectively,
the obtained surface $Z$ and $Z^\prime$ becomes $k$-biregular.
Then, except finite number of elements of $k$,
$C_a$ and $C_a^\prime$ (denoted with $C$ for simplicity)
should satisfy $C \cdot C=0$ and $C \cdot \Omega=-2$
on $Z^\prime$, where $\Omega$ is the canonical divisor.

Since by a blowing-up $E$,
$C \cdot C$ decreases by $(C \cdot E)^2$
and $C \cdot \Omega$ increases by $C \cdot E$,
we should have
\begin{equation} \label{1_2}
C \cdot C=\sum_j m_j^2, \quad
C \cdot \Omega=-2-\sum_j m_j
\end{equation}
on $X$, where $m_j=C \cdot E_j$.

On the other hand, we can prove
\begin{thm} \label{B}
If $\deg P \geq 7$,
we can construcat a non-singular projective surface $X^\prime$
which is $l$-birational with $X$,
such that for any $\{m_j\}$ any
$\mathfrak{G}$-invariant irreducible curve $C$ other than
$x=$const. can not satisfy (\ref{1_2}).
\end{thm}

Theorem \ref{B} impies that $k(x,y,z)$ is not $k$-rational when $\deg P \geq 7$.

\noindent
III. Reduction to a del Pezzo surface.

A del Pezzo surface $S$ is biregular with  some successive blowings-up of the projective plane $\mathbb{P}^2$.

\begin{thm} \label{C}
When $3 \leq \deg P \leq 6$,
if $k(x,y,z)$ is $k$-ratioanl, then we can construct a del Pezzo surface $X^{\prime\prime}$
which is $l$-birational with $X$.
\end{thm}

Thus, $X^{\prime\prime}$ is biregular with some successive blowings-up of $\mathbb{P}^2.$
From this result, we can deduce a contradiction
and complete the proof of $k$-irrationality of $k(x,y,z)$.

\subsection{Plan of this paper}

We shall prove Theorem \ref{A}, \ref{B}, \ref{C} and the required contradiction
in \S3, \S5, \S7 respectively.
Necessary background in algebraic geometry is explained in
\S2, \S4, \S6 respectively.

The coefficient field $k$ is algebraically closed in \S2, \S4 and \S6,
while $k$ is algebraically non-closed in \S3, \S5 and \S7.
The results in \S2, \S4, \S6 are applied for $\overline{k}$,
the algebraic closure of the field $k$
in \S3, \S5 and \S7.

\begin{rem}
The case when $\deg P=3$ and $P(x)$ is irreducible is a typical example of a surface
which is not $k$-rational but stably $k$-rational.
(Beauville, Colliot-Th\'el\`ene, Sansuc, Swinnerton-Dyer \cite{BCSS85}).
\end{rem}

\begin{rem}
When $\ch k=0$,
then $k$ is necessarily an infinite field.
When $\ch k>0$, $k$ may be a finite field.
Even then, without loss of generality,
we can assume that $|k|$ is sufficiently large,
by the following reason.

Let $k$ be a finite field and suppose that the conditions
(1) $\sim$ (5) are satisfied.
Then for any $N>0$, there exists a finite extension $k^\prime \supset k$
such that $|k^\prime|>N$ and the conditions (1) $\sim$ (5)
are satisfied even if we replace $k$ by $k^\prime$.
Note that $k^\prime$-irrationality of $k^\prime(x,y,z)$
implies $k$-irrationality of $k(x,y,z)$.

This means that if a proposition holds except finite number
of elements of $\overline{k}$,
then it holds for some element of $k$.
\end{rem}

\section{Birational mapping and blowing-up.}

In this section, we shall state some results in algebraic geometry
without proof, necessary for the discussions in \S3.
For more details, see for instance Hartshorn \cite{Har77}, especially chap. 5.

\subsection{Birational mapping} \label{birational}

Let $X$ and $X^\prime$ be projective non-singular surfaces,
which are mutually birational by $T: X \rightarrow X^\prime$.
$T$ can not be defined for finite number of points
(which are called fundamental points of $T$)
because the both of numerator and denominator of $T$ becomes zero.
$T$ is not injective on finite number of irreducible curves
(which are called exceptional curves of $T$),
and $T$ maps every irreducible branch of exceptional curves to a point of $X^\prime$,
which is a fundamental point of $T^{-1}$.

The complement $O(X)$ of all fundamental points and all exceptional curves
is a Zariski open set of $X$,
and $T$ maps $O(X)$ biregularly to $O(X^\prime)$ (defined similarly for $T^{-1}$).

\begin{thm}
(Proof is omitted, see Hartshorn \cite{Har77} Chap. 5)

Every birational mapping $T: X \mapsto X^\prime$
becomes biregular after finite steps of blowing-up at fundamental points
of $T$ and $T^\prime$ respectively.
(This is valid only for surfaces, and it is not true for higher dimensional varieties).
\end{thm}

A concrete example of such blowings-up is given in the discussion in the subsection \ref{biregularization}.
 
\subsection{Blowing-up}

Let $X$ be a projective non-singular surface, and $P$ be a point on $X$.
Then, there exists uniquely (modulo biregularity) a projective non-singular surface
$\widetilde X$ which satisfies the followings.
$\widetilde X$ is called the blowing-up of $X$ at $P$.

$X$ and $X^\prime$ are mutually birational by $\pi: \widetilde X \mapsto X$,
and
\begin{sub}
\item[(1)]
$\pi$ is regular and has no fundamental point.
$\pi$ has a unique exceptional curve $E_p$,
which is biregular to the projective line $\mathbb{P}^1$,
and $\pi$ maps $E_P$ to $P$.
\item[(2)]
$\pi^{-1}$ has a unique fundamental point $P$
and has no exceptional curve.
\end{sub}
In other words, $X \setminus \{P\}$
and $X^\prime \setminus E_P$ are mapped biregularly
and $\pi$ maps $E_P$ to $P$
while $\pi^{-1}$ is not defined at $P$.

Roughly speaking, $\widetilde X$ is the dilation of a point $P$ to a line $E_P$ in $\widetilde X$.
In the tangent plane of $X$ at $P$,
the direction ratioes of tangent vectors correspond to
points on $E_P$.
Thus $E_P$ is the set of direction ratioes of tangent vectors at $P$.

\subsection{Div($X$) and Pic($X$).}

Let $X$ be a projective non-singular surface.
The divisor group $Div(X)$ is defined as the free $\mathbb{Z}$-module
with all irreducible curves on $X$ as basis.
Every irreducible curve $C$ on $X$ induces a valuation $v_C$
on the function field $k(X)$,
and for $f \in k(X)$, the divisor $\sum v_C(f)C$
is called the principal divisor of $f$.

When $f$ runs over $k(X)$,
the principal divisors form a subgroup of $Div(X)$,
which are called the principal divisor group.
It is isomorphic to $k(X)^\times/k^\times$.

The factor group of $Div(X)$ by the principal divisor group is called
the divisor class group or Picard group and denoted with $Pic(X)$. 

\begin{rem}
In more general setting,
the definition of Picard group is more complicated,
but for a projective non-singular surface,
it is nothing but the divisor class group.
\end{rem}

Let $\widetilde X$ be the blowing-up of $X$ at $P$.
Let $C$ be an irreducible curve on $X$.
If $C$ does not pass through $P$,
then $\widetilde C:=\pi^{-1}(C)$
is an irreducible curve on $X$.
If $C$ passes through $P$,
let $\widetilde C$ be the Zariski closure of
$\pi^{-1}(C \setminus \{P\})$
in $\widetilde X$,
then $\widetilde C$ is an irreducible curve on $\widetilde X$.
Besides $\widetilde C$, the only one irreducible curve
on $\widetilde X$ is $E_P$.
So idetifying $C$ and $\widetilde C$,
we have $Div(\widetilde X)=Div(X) \oplus \mathbb{Z}$,
where $\mathbb{Z}$-part is the free $\mathbb{Z}$-module
with $E_P$ as the base.

Since $X$ and $\widetilde X$ are birational,
the function fields are the same,
$k(\widetilde X)=k(X)$.
Taking the factor group by the common principal divisor group,
we have $Pic(\widetilde X) \simeq Pic(X) \oplus \mathbb{Z}$,
where $\mathbb{Z}$-part is the free $\mathbb{Z}$-module
with $E_P$ as the base.

We shall give the isomorphism more explicitly in the next section,
using the intersection forms.

\subsection{Intersection form.} \label{intersection}

\begin{thm}
(Proof is omitted, see Hartshorn \cite{Har77} Chap. 5)

On $Div(X) \times Div(X)$,
there exists uniquely a symmetric $\mathbb{Z}$-bilinear form
$D_1 \cdot D_2$ satisfying the following conditions.
It is colled the intersection form.
\begin{sub}
\item[(1)]
If two irreducible curves $C_1$ and $C_2$ do not intersect on $X$,
then $C_1 \cdot C_2 = 0$.
\item[(2)]
If $C_1$ and $C_2$ intersects transversally at $n$ points,
then $C_1 \cdot C_2=n$.
Here ``intersects transversally at $P$'' 
means that both $C_1$ and $C_2$ are non-singular at $P$,
and tangent vectors of $C_1$ and $C_2$ at $P$ are linearly independent.
\item[(3)]
If $D$ is a principal divisor,
then $D \cdot D^\prime=0$ for all $D^\prime \in Div(X)$.
So that the intersection form is defined on $Pic(X) \times Pic(X)$.
\end{sub}
\end{thm}
If $C_1$ and $C_2$ intersect at $n$ points,
but not transversely at some point,
then we have $C_1 \cdot C_2 >n$.
So, for every two different irreducible curves $C_1$, $C_2$,
we have $C_1 \cdot C_2 \geq 0$.
But $C \cdot C$ (called the self-intgersection number of $C$)
can be $<0$.
Note that $C \cdot C$ is determined indirectly using the conditoin (3).

The relation of the intersection form and blowing-up is as follows.

First, consider $E_P \cdot \widetilde C$.
From (1) and (2) above, we have
\begin{sub}
\item[($1^\prime$)]
If $C$ does not pass through $P$,
then $E_P \cdot \widetilde C=0$.
\item[($2^\prime$)]
If $C$ passes through $P$,
and $C$ is non-singular at $P$,
then $E_P \cdot \widetilde C=1$.
\item[($3^\prime$)]
Suppose that $C$ passes through $P$,
and $C$ is singular at $P$.
The local equation of $C$ is given by $F(x,y)=0$
where $x,y$ is a local coordinate at $P$
with $x=y=0$ at $P$,
and $F(x,y)$ is a formal power series of $x$ and $y$.
Since $C$ passes through $P$,
the constant term of $F$ is zero.
Since $C$ is singular at $P$,
the coefficients of $x$ and $y$ are also zero.
Let $\nu$ be the smallest integer of $i+j$
such that the coefficient of $x^iy^j$ is not zero,
then $E_P \cdot \widetilde C=\nu$.
\end{sub}

Note that the homogeneous part of degree $\nu$ of $F$ induces a polynomial of degree $\nu$
in $\frac{y}{x}$,
so there are $\nu$ roots of $\frac{y}{x}$.

Using this $E_P \cdot \widetilde C$, we have
\begin{equation} \label{2_1}
\widetilde C_1 \cdot \widetilde C_2=C_1 \cdot C_2 -
(E_P \cdot \widetilde C_1)(E_P \cdot \widetilde C_2).
\end{equation}
For simplicity, suppose that $\widetilde C_1$ and $\widetilde C_2$ do not intersect on $E_P$.
Since $C_1$ passes $(E_P \cdot \widetilde C_1)$ times through $P$
and $C_2$ passes $(E_P \cdot \widetilde C_2)$ times through $P$,
there are $(E_P \cdot \widetilde C_1)(E_P \cdot \widetilde C_2)$
virtual intersection points on $X$.
This verifies the formula (\ref{2_1}).
More consideratoins show that the above formula (\ref{2_1}) is valid
even if $\widetilde C_1$ and $\widetilde C_2$ intersect on $E_P$.
Finally we have
\begin{equation}
E_P \cdot E_P=-1,
\end{equation}
which is obtained using the condition (3).

Considering the valuation $v_{{}_C}$, we see that if $D$ is a principal divisor on $X$,
then $\widetilde D+(E_P \cdot \widetilde D)E_P$ is a principal divisor on $\widetilde X$.
This derives the following fact.

Let $\pi^\ast$ be a $\mathbb{Z}$-linear map $Div(X) \rightarrow Div(\widetilde X)$
defined by $\pi^\ast(D)=\widetilde D+(E_P \cdot \widetilde D)E_P$.
Then $\pi^\ast$ is injective and maps the principal divisor group
to the principal divisor group.
So taking the factor group, we get the isomorphism
$Pic(\widetilde X) \simeq Pic(X) \oplus \mathbb{Z}$.
 
Even if $D_1 \equiv D_2$ ($\equiv$ means the identity modulo principal divisor group),
\begin{equation}
\widetilde D_1 \equiv \widetilde D_2 +
\big\{(E_P \cdot \widetilde D_2)-(E_P \cdot \widetilde D_1)\big\}E_P.
\end{equation}

\subsection{Biregularization of $T$.} \label{biregularization}

We shall return to our problem.
Let $T$ be the birational transformation of $\mathbb{P}^1 \times \mathbb{P}^1$
defined by $T: x \mapsto x, u \mapsto \frac{P(x)}{u}$.
Let $r=\deg P$ and $c_1,c_2,\dots,c_r$ be the roots of $P$.
$T$ has $r+1$ fundamental points and $r+1$ exceptional curves.
Fundamental points are $P_i:x=c_i, u=0 \, (1 \leq i \leq r)$ and $P_{r+1}:x=u=\infty$.
Exceptional curves are $x=c_i \, (1 \leq i \leq r)$ and $x=\infty$.

Consider the blowings-up for each $P_i$.
Let $X_1$ be the blowing-up of $\mathbb{P}^1 \times \mathbb{P}^1$ at $P_1$,
and let $T_1$ be the lifting of $T$ to $X_1$.
$X_1$ is a surface in $\mathbb{P}^1 \times \mathbb{P}^1 \times \mathbb{P}^1$
defined by $\frac{u}{x-c_1}=t_1$.
$E_1$ is the curve $x=c_1, u=0$,
and $\widetilde{x=c_1}$ is the curve $x=c_1, t_1=\infty$.

We write $P(x)=b\prod_{j=1}^r(x-c_j)$.
By $T_1$, each point $(c_1,u,\infty) \in (\widetilde{x=c_1})$
is mapped to $(c_1,0,\frac{b\prod_{j\not=1}(c_1-c_j)}{u}) \in E_1$,
and each point $(c_1,0,t_1) \in E_1$
is mapped to $(c_1,\frac{b\prod_{j \not= 1}(c_1-c_j)}{t_1},\infty)
\in (\widetilde{x=c_1})$.
So $T_1$ maps $E_1$ biregularly to $\widetilde{x=c_1}$.

Next let $X_2$ be the blowing-up of $X_1$ at $P_2$,
and let $T_2$ be the lifting of $T_1$ to $X_2$.
Repeat this $r$ times
so that $T_r$ maps $E_i$ biregularly to $\widetilde{x=c_i}$
for $1 \leq i \leq r$.
Now the only fundamental point of $T_r$ is $P_{r+1}: x=u=\infty$,
and the only exceptional curve of $T_r$ is $x=\infty$.

$$\begin{matrix}
X_r & \stackrel{T_r}{\longrightarrow} & X_r \\
\downarrow & & \downarrow \\
\vdots & & \vdots \\
\downarrow & & \downarrow \\
X_2 & \stackrel{T_2}{\longrightarrow} & X_2 \\
\downarrow & & \downarrow \\
X_1 & \stackrel{T_1}{\longrightarrow} & X_1 \\
\downarrow & & \downarrow \\
\mathbb{P}^1 \times \mathbb{P}^1 &
\stackrel{T}{\longrightarrow} & \mathbb{P}^1 \times \mathbb{P}^1 \\
\end{matrix}$$

The blowing-up at $P_{r+1}$ does not make $T_{r+1}$ biregular.
Let $X_{r+1}$ be the blowing-up of $X_r$ at $P_{r+1}$,
and let $T_{r+1}$ be the lifting of $T_r$ to $X_{r+1}$.
$X_{r+1}$ is a surface in $X_r \times \mathbb{P}^1$
defined by $\frac{u}{x}=t_{r+1}$.
$E_{r+1}$ is the curve $x=\infty, u=\infty$
and $\widetilde{x=\infty}$ is the curve $x=\infty, t_{r+1}=0$.

$T_{r+1}$ maps $\widetilde{x=\infty}$ to
one point $P_{r+2} \in E_{r+1}$ defined by $t_{r+1}=\infty$, 
and maps $E_{r+1}$ to $P_{r+2}$.
So exceptional curves of $T_{r+1}$ are
$\widetilde{x=\infty}$ and $E_{r+1}$
while the only fundamental point is $P_{r+2}$.

So, blow up again.
Let $X_{r+2}$ be the blowing-up of $X_{r+1}$ at $P_{r+2}$,
and let $T_{r+2}$ be the lifting of $T_{r+1}$ to $X_{r+2}$.
$X_{r+2}$ is a surface in $X_{r+1} \times \mathbb{P}^1$ defined by
$\frac{t_{r+1}}{x}=t_{r+2}$.
Then $T_{r+2}$ maps $\widetilde{x=\infty}$ to one point $P_{r+3} \in E_{r+2}$
defined by $t_{r+2}=\infty$.
If $r=3$, then $T_5$ maps $E_4$ biregularly to $E_5$,
but if $r>3$, then $T_{r+2}$ maps both of $E_{r+1}$ and $E_{r+2}$ to one point $P_{r+3}$.

Repeating this $r-1$ times,
the only exceptional curve is $\widetilde{x=\infty}$ and
the only fundamental point is $P_{2r}$.
Let $X$ be the blowing-up of $X_{2r-1}$ at $P_{2r}$.
$X$ is a surface in $X_{2r-1} \times \mathbb{P}^1$ defined by
$\frac{t_{2r-1}}{x}=t_{2r}$.
Then $T_{2r}$ becomes biregular,
namely $T_{2r}$ maps $\widetilde{x=\infty}$ to $E_{2r}$,
and $E_{r+i}$ to $E_{2r-i}$ $(1 \leq i \leq r-1)$ biregularly.

$$\begin{matrix}
X & \stackrel{T_{2r}}{\longrightarrow} & X \\
\downarrow & & \downarrow \\
\vdots & & \vdots \\
\downarrow & & \downarrow \\
X_{r+2} & \stackrel{T_{r+2}}{\longrightarrow} & X_{r+2} \\
\downarrow & & \downarrow \\
X_{r+1} & \stackrel{T_{r+1}}{\longrightarrow} & X_{r+1} \\
\downarrow & & \downarrow \\
X_r & \stackrel{T_r}{\longrightarrow} & X_r \\
\end{matrix}$$

Thus, $T$ becomes biregular after $2r$ blowings-up in total,
once for each $1 \leq i \leq r$, and $r$ times for $P_{r+1}$.
We denote the obtained surface with $X$.

$Pic(\mathbb{P}^1 \times \mathbb{P}^1)$ has rank 2,
with $x=\infty$ and $u=\infty$ as the basis.
So that $Pic(X)$ has rank $2r+2$ with $x=\infty, u=\infty$,
and $E_i \, (1 \leq i \leq 2r)$ as the basis.

\section{Manin's criterion} \label{Manin}

\subsection{$Pic(X)$ as a Galois module.}

Let $k$ be an algebraically non-closed field,
and $K$ be an algebraic function field with two variables over $k$.
Nameky, $K$ is a finite extension of the rational funciton field with two variables over $k$
such that $k$ is algebraically closed in $K$.

Let $\overline{k}$ be the algebraic closure of $k$.
The $k$-automorphism group of $\overline{k}$ is isomorphic to
$Gal(k_{sep}/k)$, where $k_{sep}$ is the separable closure of $k$,
because every $k$-automorphism of $k_{sep}$ is extended
uniquely to $\overline{k}$.
$G:=Gal(k_{sep}/k)$ acts on
$\overline{k} \otimes_k K$,
assuming that it acts on $K$ trivially,
namely $G \ni \sigma \mapsto \overline{\sigma}=\sigma \otimes id_K$.

Assume that $\overline{k} \otimes_k K$ is rational over $\overline{k}$,
namely $\overline{k} \otimes_k K=\overline{k}(u,v)$ for some $u,v$.
Let $u^\sigma, v^\sigma$ be the image of $u,v$ by the action of $\overline{\sigma}$,
then we have $\overline{k}(u,v)=\overline{k}(u^\sigma,v^\sigma)$,
so that $u \mapsto u^\sigma, v \mapsto v^\sigma$
induces a $\overline{k}$-automorphism $T_\sigma$
of $\overline{k}(u,v)$.
$T_\sigma$ is different from $\overline{\sigma}$,
because $T_\sigma$ acts trivially on $\overline{k}$.
Let $\widetilde \sigma$ be a $\overline{k}$-automorphism of $\overline{k}(u,v)$
such that $\widetilde \sigma$ acts naturally on $\overline{k}$,
and acts trivially on $u$ and $v$.
Then we have $\overline{\sigma}=T_\sigma \circ \widetilde \sigma$.

$T_\sigma$ induces a birational transformation of $\mathbb{P}^1 \times \mathbb{P}^1$,
while $\widetilde \sigma$ induces a homeomorphic transformation
in Zariski topology of $\mathbb{P}^1 \times \mathbb{P}^1$.
Since both $u$ and $v$ belong to $l \otimes_k K$
for sufficiently large finite extension $l$ of $k$,
the number of different $T_\sigma$ is finite,
so that after finite steps of blowings-up of $\mathbb{P}^1 \times \mathbb{P}^1$,
all of $T_\sigma$ become biregular on the obtained surface $X$.
The lifting of $\widetilde \sigma$ to $X$ is homeomorphic in Zariski topology.
So the action of $\overline{\sigma}$ induces a permutation of irreducible curves,
and $Div(X)$ becomes a permutation $G$-module.

Since the action of $\overline{\sigma}$ keeps the function field
$\overline{k} \otimes_k K=\overline{k}(u,v)$ invariant,
it keeps the principal divisor group invariant,
so taking the factor module,
we see that $Pic(X)$ is also a $G$-module.

But since $P(X)$ is of finite rank as a $\mathbb{Z}$-module,
and since every orbit of an irreducible curve by the action of $G$ is finite,
$Pic(X)$ is actually a $\mathfrak{G}$-module,
where $\mathfrak{G}=Gal(l/k)$,
$l$ being a sufficiently large finite Galois extension of $k$.

Thus, $Pic(X)$ becomes a $\mathfrak{G}$-lattice.
Here a $\mathfrak{G}$-lattice means a free $\mathbb{Z}$-module of finite rank
with the action of $\mathfrak{G}$ as automorphisms.

\subsection{$k$-rationality and Galois cohomology.}

Let $K^\prime$ be another algebraic function field with two variables over $k$
such that $\overline{k} \otimes_k K^\prime$ is $\overline{k}$-rational.
Let $\overline{k} \otimes_k K^\prime=\overline{k}(u^\prime,v^\prime)$.
$G=Gal(k_{sep}/k)$ acts on $\overline{k} \otimes_k K^\prime$
as $G \ni \sigma \mapsto \overline{\sigma}^\prime=\sigma \otimes id_{K^\prime}$.

By the discussions in the previous subsection,
$\overline{\sigma}^\prime$ can be wirtten as
$\overline{\sigma}^\prime=T_\sigma^\prime \circ \widetilde{\sigma}^\prime$,
where $T_\sigma^\prime$ is a $\overline{k}$-automorphism of
$\overline{k}(u^\prime,v^\prime)$ and $\widetilde{\sigma}^\prime$
is a $\overline{k}$-automorphism of $\overline{k}(u^\prime,v^\prime)$
which acts on $\overline{k}$ naturally and asts on $u^\prime$ and $v^\prime$ trivially.
$T_\sigma^\prime$ induces a birational transformation of
$\mathbb{P}^1 \times \mathbb{P}^1$,
but after finite steps of successive blowings-up,
all $T_\sigma^\prime$ becomes biregular on the obtained surface $X^\prime$,
so we can regard $Pic(X^\prime)$ as a $\mathfrak{G}$-lattice,
where $\mathfrak{G}=Gal(l/k)$ for sufficiently large finite Galois extension $l$ of $k$.

\begin{prop}
$K$ is $k$-isomorph with $K^\prime$,
if and only if there exists a $\overline{k}$-isomorphism $T$
from $\overline{k} \otimes_k K$ to $\overline{k} \otimes_k K^\prime$
which commutes with the action of $G$,
namely for $\forall \sigma \ni G=Gal(k_{sep}/k),
T \circ \overline{\sigma}=\overline{\sigma}^\prime \circ T$.
\end{prop}

\begin{proof}
Suppose that $K$ is $k$-isomorph with $K^\prime$ and let $T_0$ be the
$k$-isomorphism.
Then $T_0$ is naturally extended to a $\overline{k}$-isomorphism $T$
from $\overline{k} \otimes_k K$ to $\overline{k} \otimes_K K^\prime$,
$T=id_{\overline{k}} \otimes T_0$.
Evidently $T$ commutes with the action of $G$.

Conversely, suppose that a required $\overline{k}$-isomorphism $T$ exists.
Since $T$ commutes with the action of $G$,
$T$ and $T^{-1}$ map the fixed field of $G$ to each other.
However, the fixed field of $\overline{k} \otimes_k K$ 
(resp. $\overline{k} \otimes_k K^\prime$) of $G$ is $K$ (resp. $K^\prime$),
and the restriction of $T$ on $K$ becomes a $k$-isomorphism from $K$ to $K^\prime$.
\end{proof}

\begin{prop}
(Manin \cite{Man67})
If $K$ is $k$-isomorph with $K^\prime$,
then $Pic(X)$ is similar with $Pic(X^\prime)$.
This means that there exist permutation $\mathfrak{G}$-lattices
$P_1$ and $P_2$ such that $Pic(X) \oplus P_1 \simeq Pic(X^\prime) \oplus P_2$.
\end{prop}

\begin{proof}
Assume the existence of a required $\overline{k}$-isomorphism $T$
from $\overline{k}(u,v)$ to $\overline{k}(u^\prime,v^\prime)$.
Then $T$ induces a birational transformation of $\mathbb{P}^1 \times \mathbb{P}^1$.
After suitable blowings-up,
$T$ is lifted to a birational map from $X$ to $X^\prime$.
Though it may not be biregular on $X$,
after further suitable blowings-up,
we can reach the surface $Z$ and $Z^\prime$,
on which $T$ (and $T^{-1}$) becomes biregular.

Since $T$ is biregular, we have $Pic(Z) \simeq Pic(Z^\prime)$
as $\mathbb{Z}$-modules.
Since $T$ comutes with the action of $G$,
(then their liftings commutes also),
$Pic(Z) \simeq Pic(Z^\prime)$ as $\mathfrak{G}$-lattices also.

Only remained to prove is that $Pic(Z) \simeq Pic(X) \oplus P$
for some permutation $\mathfrak{G}$-lattice $P$.

Let $\{E_j\}$ be the successive blowings-up to reach $Z$ from $X$.
Since $T$ commutes with the action of $G$,
the set of fundamental points of $T$ is $G$-invariant,
and the action of $G$ induces permutations of $\{E_j\}$.

Let $\{e_i\}$ be the basis of $Pic(X)$ as a free $\mathbb{Z}$-module.
Then $Pic(Z)$ is a free $\mathbb{Z}$-module
with the basis $\{\pi^\ast e_i\} \cup \{ E_j \}$,
where $\pi^\ast$ is a $\mathbb{Z}$-linear map
from $Pic(X)$ to $Pic(Z)$,
obtained by the iteration of $\pi^\ast$ mentioned at the end of the subsection \ref{intersection}.
Let $M_1$ (resp. $M_2$) be a free $\mathbb{Z}$-module with the basis
$\{\pi^\ast e_i\}$ (resp. $\{E_j\}$),
then $Pic(Z) \simeq M_1 \oplus M_2$ as $\mathbb{Z}$-modules.
However, $M_2$ is a permutation $\mathfrak{G}$-lattice
as mentioned above.
We can show that $M_1$ is also a $\mathfrak{G}$-lattice isomorphic to $Pic(X)$.

This completes the proof.
\end{proof}

\begin{cor}
If $K$ is $k$-isomorph with $K^\prime$, then
$H^1(\mathfrak{G},Pic(X)) \simeq H^1(\mathfrak{G},Pic(X^\prime))$
and $\widehat H^{-1}(\mathfrak{G},Pic(X)) \simeq
\widehat H^{-1}(\mathfrak{G},Pic(X^\prime))$,
where $H^1$ is the first Galois cohomology
and $\widehat H^{-1}$ is Tate cohomology.
\end{cor}

This comes from $H^1(\mathfrak{G},P)=\widehat H^{-1}(\mathfrak{G},P)=0$
for a permutation $\mathfrak{G}$-lattice $P$.

Especially, if $K$ is $k$-rational,
then $H^1(\mathfrak{G},Pic(X))=\widehat H^{-1}(\mathfrak{G},Pic(X))=0$
by the following reason.

The $k$-rationality of $K$ means that $K$ is $k$-isomorph with
the two dimensional rational function field $K^\prime=k(x,y)$.
In this case,
$\overline{k} \otimes_k K^\prime=\overline{k}(x,y)$
and $\overline{\sigma}^\prime$ acts trivially on $x$ and $y$.
So, $X^\prime=\mathbb{P}^1 \times \mathbb{P}^1$
and $Pic(X^\prime)$ is a trivial $\mathfrak{G}$-lattice,
so that $H^1(\mathfrak{G},Pic(X^\prime))=\widehat H^{-1}(\mathfrak{G},Pic(X^\prime))=0$.
In other words, $H^1(\mathfrak{G},Pic(X)) \not=0$
or $\widehat H^{-1}(\mathfrak{G},Pic(X)) \not= 0$
is a criterion for the $k$-irrationality of $K$.

\subsection{Proof of Theorem \ref{A}}
Let $K$ be the quadratic extension of $k(x,y)$ defined by
(\ref{cha}):$z^2=ay^2+P(x)$ with conditions (1) $\sim$ (5)
in the subsection \ref{main}.
Then $\overline{k} \otimes_k K=\overline{k}(x,u)$
where $u=z+\sqrt{a}y$.
For $\sigma \in G=Gal(k_{sep}/k)$,
$T_\sigma$ is the identity or equal to $T: x \mapsto x, u \mapsto \frac{P(x)}{u}$
according to whether $\sqrt{a}$ is invariant by $\sigma$ or not.
$T$ induces a birational transformation of $\mathbb{P}^1 \times \mathbb{P}^1$,
and after $2r$ blowings-up,
it becomes biregular on the obtained surface $X$,
as studied in the section \ref{biregularization}.
Then $Pic(X)$ is a free $\mathbb{Z}$-module of rank $2r+2$
with the basis $E_i (1 \leq i \leq 2r)$ and $x=\infty, u=\infty$.

For simplicity, we omit $\widetilde{\ }$ for the blowing-up,
and the lifting of $C$ is denoted with the same symbol $C$.
The confusion is avoided by seeing $C$ is a divisor of which surface.

We shall determine the structure of $Pic(X)$ as a $\mathfrak{G}$-lattice,
where $\mathfrak{G}=Gla(l/k)$,
$l$ being the splitting field of $P(x)$.

As studied in the subsection \ref{biregularization},
$T$ maps $E_i$ to $(x=c_i)$
for $1 \leq i \leq r$,
$(x=\infty)$ to $E_{2r}$,
$E_{r+i}$ to $E_{2r-i}$ for $1 \leq i \leq r-1$,
and $(u=\infty)$ to $(u=0)$.
The next question is what divisor classes
$(x=c_i)$ and ($u=0)$ belong to.

Let $\pi^\ast$ be a $\mathbb{Z}$-linear map from $Div(\mathbb{P}^1 \times \mathbb{P}^1)$
to $Div(X)$,
obtained by the iteration of $\pi^\ast$ mentionaed at the end of the subsection \ref{intersection}.
Then we have
\begin{eqnarray}
\pi^\ast(x=c_i) &=& (x=c_i)+E_i \text{ for } 1 \leq i \leq r \nonumber \\
\pi^\ast(x=\infty) &=& (x=\infty)+\sum_{j=1}^r E_{r+j} \\
\pi^\ast(u=0) &=& (u=0) + \sum_{j=1}^r E_j \nonumber \\
\pi^\ast(u=\infty) &=& (u=\infty) + \sum_{j=1}^r jE_{r+j} \nonumber
\end{eqnarray}
In $Div(\mathbb{P}^1 \times \mathbb{P}^1)$
we have $(x=c_i) \equiv (x=\infty)$
and $(u=0) \equiv (u=\infty)$,
so that in $Div(X)$, we have
\begin{eqnarray}
(x=c_i) &\equiv& (x=\infty)-E_i+\sum_{j=1}^r E_{r+j} \\
(u=0) &\equiv& (u=\infty)-\sum_{j=1}^r E_j + \sum_{j=1}^r jE_{r+j} \nonumber
\end{eqnarray}
Therefore, the action of $\overline{\sigma}=T_\sigma \circ \widetilde{\sigma}$
on $Pic(X)$ is represented by the following matrix $g_\sigma$,
with $E_1, E_2, \dots, E_r, (u=\infty), (x=\infty), E_{r+1}, E_{r+2},\dots, E_{2r}$
as the basis in this order.
\begin{eqnarray}
\text{For } \sigma \in N=Gal(l/k(\sqrt{a})), &
g_\sigma=
\begin{pmatrix}
A_\sigma & 0 \\
0 & I_{r+2} \\
\end{pmatrix} \\
\text{For } \sigma \in \mathfrak{G} \setminus N, &
g_\sigma=
\begin{pmatrix}
-A_\sigma & -1 & 0 \\
0 & 1 & 0 \\
1 & c & B
\end{pmatrix} \nonumber
\end{eqnarray}
where $B=\begin{pmatrix}
0 & & 1 \\
 & \rotatebox{90}{ $\ddots$ } & \\
1 & & 0
\end{pmatrix}$
is a $(r+1) \times (r+1)$ matrix,
$c=(0,1,2,\dots,r)$,
and 1 (resp. 0, -1) stands for the matrix whose entries are all 1
(resp. 0,-1).

$A_\sigma$ is the permutation matrix of the permutation of $\{c_i\}$
induced by $\sigma$.
Suppose that $P(x)$ is a product of $r^\prime$ irreducible polynomials.
Then, the set of roots $\{c_i\}$ of $P(x)$ is divided into $r^\prime$ blocks,
each of which consists of the roots of the same irerducible component.
Each block is a transitive part by the action of $\mathfrak{G}$.
Since each irreducible component is assumed to be irreducible
also over $k(\sqrt{a})$,
the action of $N$ is also transitive on each block.
The block is called even (resp. odd),
when the degree of the corresponding irreducible polynomial is even (resp. odd).
$Pic(X)$ is decomposed into the direct sum of two $\mathfrak{G}$-lattices,
one of which is a permutation $\mathfrak{G}$-lattice.
\begin{equation}
Pic(X)=M_1 \oplus M_2,
M_2 \text{ is a permutation } \mathfrak{G} \text{-lattice.}
\end{equation}
The rank of $M_2$ is $r+1$ when $r$ is odd,
and $r$ when $r$ is even.
Its basis are $(x=\infty), E_{r+1}, \dots, E_{2r}$
when $r$ is odd,
omitting $E_{\frac{3r}{2}}$ when $r$ is even.

The rank of $M_1$ is $r+1$ when $r$ is odd,
and $r+2$ when $r$ is even.
Its basis are the following $\{e_i\}$.
\begin{eqnarray}
e_i=E_i-(x=\infty)-\sum_{1 \leq j < \frac{r}{2}}E_{r+j} \text{ for } 1 \leq i \leq r \nonumber \\
e_{r+1}=(u=\infty)-r(x=\infty)-\sum_{1 \leq j < \frac{r}{2}}(r-j)E_{r+j}. \\
\text{when } r \text{ is even, } e_{r+2}=E_{\frac{3r}{2}} \nonumber
\end{eqnarray} 
The action of $\mathfrak{G}$ on $M_1$ is representaed by the following matrices
with the basis $\{e_i\}$.
(We use the same symbol $g_\sigma$).
\hfill\break
When $r$ is odd
\begin{eqnarray} \label{odd1}
\text{For } \sigma \in N=Gal(l/k(\sqrt{a})), & g_\sigma=
\begin{pmatrix} A_\sigma & 0 \\ 0 & 1 \end{pmatrix} \\
\text{For } \sigma \in \mathfrak{G} \setminus N, & g_\sigma=
\begin{pmatrix} -A_\sigma & -1 \\ 0 & 1 \end{pmatrix} \nonumber
\end{eqnarray}
\hfill\break
When $r$ is even
\begin{eqnarray} \label{even1}
\text{For } \sigma \in N=Gal(l/k(\sqrt{a})), & g_\sigma=
\begin{pmatrix} A_\sigma & 0 \\ 0 & I_2 \end{pmatrix} \\
\text{For } \sigma \in \mathfrak{G} \setminus N, & g_\sigma=
\begin{pmatrix} -A_\sigma & -1 & 0 \\
0 & 1 & 0 \\
1 & r/2 & 1 \end{pmatrix} \nonumber
\end{eqnarray}

Since $M_2$ is a permutataion $\mathfrak{G}$-lattice,
we have $\widehat H^{-1}(\mathfrak{G},Pic(X))
\simeq \widehat H^{-1}(\mathfrak{G},M_1)$ and
$H^1(\mathfrak{G},Pic(X)) \simeq H^1(\mathfrak{G},M_1)$.

Let $M_0$ be the submodule of $M_1$ spanned by
$\{e_i | 1 \leq i \leq r\}$.
An element of $M_0$ is written as $\sum_{i=1}^r a_ie_i$, $a_i \in \mathbb{Z}$.
Let $s_j$ be the sum of $a_i$
when $i$ runs over the $j$-th block.
Let $M_e$ be the submodule of $M_0$,
consisting of elements such that
$\sum_{j=1}^{r^\prime} s_j$ is even.
Let $M_b$ be the submodule of $M_0$,
consisting of elemetns such that $s_j$ is even for every $j$.
We have $M_0/M_e \simeq \mathbb{Z}/2\mathbb{Z},
M_0/M_b \simeq (\mathbb{Z}/2\mathbb{Z})^{r^\prime}$ and
$M_e/M_b \simeq (\mathbb{Z}/2\mathbb{Z})^{r^\prime-1}$,
where $r^\prime$ is the number of the blocks.

By definition, we have
$\widehat H^{-1}(\mathfrak{G},M_1)=Z/B$,
where $Z$ and $B$ are submodules of $M_1$ defined by
\begin{eqnarray}
Z=ker(\sum_{\sigma} g_\sigma) \\
B \text{ is the module spanned by }
\bigcup_{\sigma} Im(\sigma-id), \nonumber
\end{eqnarray}
where the summation and the union are taken over $\sigma \in \mathfrak{G}$.

First, we suppose that $r$ is odd.
$\sum g_\sigma$ is zero except the last column,
so that $Z=M_0$.
$B$ is spanned by $e_i+e_j, e_i-e_j$
for $i,j$ in the same block,
and one more element $\sum_{i=1}^r e_i$,
so that $B$ is generated by $M_b$
and $\sum_{i=1}^r e_i$.
Since $M_0/M_b \simeq (\mathbb{Z}/2\mathbb{Z})^{r^\prime}$
and since $\sum_{i=1}^r e_i \not\in M_b$,
we have
\begin{equation}
\widehat H^{-1}(\mathfrak{G},Pic(X)) \simeq (\mathbb{Z}/2\mathbb{Z})^{r^\prime-1}
\end{equation}

Next, we suppose that $r$ is even.
$\sum g_\sigma$ is zero
except the last row and the $(r+1)$-th column,
so that the rank of $Z$ is $r$ and the projection to $M_0$
(projection as $\mathbb{Z}$-modules) is injective.
Let $Z^\prime$ and $B^\prime$ be the image of the projection of $Z$ and $B$
respectively,
then $Z/B \simeq Z^\prime/B^\prime$.

We see that $Z^\prime=M_e$ and $B^\prime$ is generated by $M_b$ and
$\sum_{i=1}^r e_i$.
SInce $M_e/M_b \simeq (\mathbb{Z}/2\mathbb{Z})^{r^\prime-1}$,
and since $\sum_{i=1}^r e_i \in M_b$
if and only if odd block does not exist,
we have
\begin{eqnarray}
\widehat H^{-1}(\mathfrak{G},Pic(X)) \simeq (\mathbb{Z}/2\mathbb{Z})^{r^\prime-1}
& \text{ if odd block does not exist} \\
\widehat H^{-1}(\mathfrak{G},Pic(X)) \simeq (\mathbb{Z}/2\mathbb{Z})^{r^\prime-2}
& \text{ if odd blocks exist} \nonumber
\end{eqnarray}
As for $H^1(\mathfrak{G},M_1)$,
we proceed as follows.
In general for a $\mathfrak{G}$-lattice $M$,
$H^1(\mathfrak{G},M)$ is isomorphic to $\widehat H^{-1}$
of the dual lattice $M^\prime$.
So that as for $H^1(\mathfrak{G},M_1)$,
it suffices to calculate $\widehat H^{-1}$
for the transposed matrices of (\ref{odd1}) or (\ref{even1}).
The calculation shows that
$H^1(\mathfrak{G},M_1) \simeq \widehat H^{-1}(\mathfrak{G},M_1)$,
though the matrices of (\ref{odd1}) or (\ref{even1}) are not symmetric.

Thus, the proof of Theorem \ref{A} has been completed.
Therefore, the main result (stated in the subsection \ref{main})
has been proved except when $P(x)$ is irreducible or
a product of two irreducible polynomials of odd degree.

\section{Canonical divisor and blowing-down.}

\subsection{Canonical divisor}

Let $X$ be a projective non-singular surface.
A canonical divisor of $X$ is defined as follows.

Let $f,g \in k(X)$ be mutually algebraic indepedent.
Let $C$ be an irreducible curve on $X$ and $P$ be a non-singular point of $C$.
Take a local coordinate $(x,y)$ at $P$
and consider the Jacobian
$\frac{\partial (f,g)}{\partial (x,y)}=
\begin{vmatrix}
\frac{\partial f}{\partial x} & \frac{\partial f}{\partial y} \\
\frac{\partial g}{\partial x} & \frac{\partial g}{\partial y} \\
\end{vmatrix}$,
then we can show that
$v_{{}_C}\Big(\frac{\partial (f,g)}{\partial (x,y)}\Big)$
is independent of the choice of a point $P$
and the choice of a local coordinate $(x,y)$.
Canonical divisor of $(f,g)$
is defined as $\sum v_{{}_C}\Big(\frac{\partial (f,g)}{\partial (x,y)}\Big)C$.

Take another $f_1,g_1 \in k(X)$ mutually algebraic independent.
Then canonical divisor of $(f_1,g_1)$ belongs to the same divisor class
with that of $(f,g)$,
namely all canonical divisors determine the unique divisor class in $Pic(X)$.
This is called the canonical divisor class of $X$ and denoted with $\Omega$.

\begin{rem}
In more general setting, the definition of the canonical divisor calss is more complicated,
but for a projective non-singular surface $X$,
it is nothing but the one defined above.
\end{rem}

\begin{exmp}
For $\mathbb{P}^1 \times \mathbb{P}^1$,
we shall determine the intersection form and the canonical divisor.

On $\mathbb{P}^1 \times \mathbb{P}^1$,
any irreducible curve $C$ other than $(x=\infty)$
and ($u=\infty)$ is the zero point set of an irreducible polynomial
$f(x,u)$.
When the degree of $f$ is $n$ with respect to $x$,
and $m$ with respect to $u$,
then we have
$C \equiv n(x=\infty)+m(u=\infty)$.

Intersection form on $\mathbb{P}^1 \times \mathbb{P}^1$
is defined by
\begin{eqnarray}
(x=\infty) \cdot (x=\infty)=(u=\infty) \cdot (u=\infty)=0 \\
(x=\infty) \cdot (u=\infty)=1 \nonumber
\end{eqnarray}
The canonical divisor is
\begin{equation} \label{ex}
\Omega=-2(x=\infty)-2(u=\infty)
\end{equation}
Take $f=x$ and $g=u$,
then since $(x,u)$ is a local coordinate except on the lines
$(x=\infty)$ and $(u=\infty)$,
we have $\frac{\partial (x,u)}{\partial (x,u)}=1$.
In a neighborhood of the line $(x=\infty)$,
a local coordinate is $(t,u)$
where $t=\frac{1}{x}$,
so $x=\frac{1}{t}$,
then $\frac{\partial (x,u)}{\partial (t,u)}=-\frac{1}{t^2}$,
thus $v_{(x=\infty)}\Big(\frac{\partial (x,u)}{\partial (t,u)}\Big)=-2$.
The similar result holds for the line $(u=\infty)$.
This verifies (\ref{ex}).
From (\ref{ex}) we see that
\begin{equation}
C \cdot \Omega =-2(m+n), \,
\Omega \cdot \Omega=8
\end{equation}
\end{exmp}

Return to a general $X$ and we shall consider the relation with the blowing-up.
Let $\widetilde X$ be the blowing-up of $X$ at a point $P$.
Then the canonical divisor of $\widetilde X$ is given by
\begin{equation} \label{OmegaX}
\Omega_{\widetilde X}
=\pi^\ast \Omega_X+E_P.
\end{equation}
This can be derived as follows.
Let $f=x$ and $g=y$,
where $(x,y)$ is a local coordinate of $X$ at $P$
with $x=y=0$ at $P$.
Since $\frac{\partial (f,g)}{\partial (x,y)}=1$,
$\Omega$ does not pass through $P$,
so $\widetilde \Omega \cdot E_p=0$.
On the other hand,
a local coordinate of $\widetilde X$ in a neighborhood of $E_p$
is $(x,t)$ where $t=\frac{y}{x}$,
so $y=tx$, then $\frac{\partial(x,y)}{\partial(x,t)}=x$,
thus $v_{E_P}\big(\frac{\partial(x,y)}{\partial(x,t)}\big)=1$.
This implies $\Omega_{\widetilde{X}}=\widetilde{\Omega}_X+E_P$.

For other canonical divisors,
extending the above relation in the form compatible with
the aciton of $\pi^\ast$,
we get (\ref{OmegaX}) above.

Since $\pi^\ast C_1 \cdot \pi^\ast C_2 =C_1 \cdot C_2$
and $\pi^\ast C \cdot E_P=0$ for any irreducible curve
$C, C_1, C_2$ on $X$,
from (\ref{OmegaX}) we have
\begin{eqnarray} \label{widetildeC}
\widetilde{C} \cdot \Omega_{\widetilde X} = C \cdot \Omega_X
+ \widetilde C \cdot E_P \\
E_P \cdot \Omega_{\widetilde X}=-1, \,
\Omega_{\widetilde X} \cdot \Omega_{\widetilde X}=
\Omega_X \cdot \Omega_X-1 \nonumber
\end{eqnarray}

\subsection{Blowing down}

Blowing-down is the inverse operation of the blowing up.
Let $X$ be a projective non-singular surface,
and assume that there eixsts an irreducible curve $F$ on $X$
satisfying $F \cdot F=-1$
and $\Omega \cdot F=-1$.
($F$ is necessarily biregular to the projective line $\mathbb{P}^1$).

\begin{thm}
There exists a unique (modulo biregularity)
projective non-singular surface $\overline{X}$
such that the blowing-up $\widetilde{\overline{X}}$
at some point $Q \in \widetilde X$
is biregular to $X$,
mapping $E_Q$ to $F$.
\end{thm}

The surface $\overline{X}$ is called the blowing-down of $X$ by $F$.
Let $\varphi$ be the biregular mapping $X \rightarrow
\widetilde{\overline{X}}$,
and $\pi$ be the projection
$\widetilde{\overline{X}} \mapsto \overline{X}$.
For an irreducible curve $C \not= F$ on $X$,
let $\overline{C}$ be the image of $C$ by $\pi \circ \varphi$,
then $\overline{C}$ is an irreducible curve of $\overline{X}$
and all irreducible curves on $\overline{X}$
are obtained in this way.
So that identifying $C$ with $\overline{C}$,
we get $Div(X)=Div(\overline{X}) \oplus \mathbb{Z}$,
where $\mathbb{Z}$-part is the free $\mathbb{Z}$-module
with $F$ as the basis.

Let $\overline{\pi}$ be the $\mathbb{Z}$-linear map from $Div(X)$ to $Div(\overline{X})$
defined by $\overline{\pi}(D)=\overline{D-\lambda F}$,
where $\lambda$ is the coefficient of $F$ in $D$.
Then $\overline{\pi}$ is surjective and 
maps the principal divisor group
to the principal divisor group bijectively.
The kernel of $\overline{\pi}$
is the free $\mathbb{Z}$-module with $F$ as the basis.
So $\overline{\pi}$ induces the isomorphism
$Pic(X) \simeq Pic(\widetilde X) \oplus \mathbb{Z}$.

Intersection form on $\overline{X}$ is given by
\begin{equation}
\overline{D}_1 \cdot \overline{D}_2=
D_1 \cdot D_2 + (D_1 \cdot F)(D_2 \cdot F).
\end{equation}
The canonical divisor of $\overline{X}$ is given by
\begin{equation}
\Omega_{\overline{X}}=
\overline{\pi}(\Omega_X)=\overline{\Omega_X-\lambda F}.
\end{equation}
We have
\begin{eqnarray}
\overline{D} \cdot \Omega_{\overline{X}}=D \cdot \Omega_X - D \cdot F \\
\Omega_{\overline{X}} \cdot \Omega_{\overline{X}}=
\Omega_X \cdot \Omega_X+1. \nonumber
\end{eqnarray}

\subsection{Blowing-up and down.} \label{updown}

Let $X$ be a projective non-singular surface and $F$ be an irreducible curve
on $X$ satisfying $F \cdot F=0$ and $F \cdot \Omega=-2$.
($F$ is necessarily biregular to the projective line $\mathbb{P}^1$).
Consider the blowing-up $\widetilde X$ at a point $P$ on $F$.
Then we have $\widetilde F \cdot \widetilde F=-1$
and $\widetilde F \cdot \Omega_{\widetilde X}=-1$,
so that we can consider the blowing-down of $\widetilde X$ by $\widetilde F$
and obtain $\overline{\widetilde{X}}$.

$X$ and $\overline{\widetilde X}$ are birational,
but not reguar in any direction.
Let $\pi_1$ be the projection $\widetilde X \rightarrow X$
and $\pi_2$ be the projection $\widetilde{\overline{\widetilde{X}}}
\rightarrow \overline{\widetilde{X}}$,
then $\rho=\pi_2 \circ \varphi \circ \pi_1^{-1}$
is the birational mapping from $X$ to $\overline{\widetilde{X}}$.

The fundamental point of $\rho$ is $P$,
and the exceptional curve of $\rho$ is $F$.
On the other hand,
the fundamental point of $\rho^{-1}$ is $Q$,
and the exceptional curve of $\rho^{-1}$ is $\overline{E}_P$.
($Q$ is a point on $\overline{E}_P$,
because $\widetilde{\overline{E}}_P \cdot E_Q=
E_P \cdot \widetilde F=1$).

For an irreducible curve $C \not=F$ on $X$,
$\overline{\widetilde{C}}$ is an irreducible curve on $\overline{\widetilde{X}}$,
and besides them,
$\overline{E}_P$ is the only irreducible curve on $\overline{\widetilde{X}}$.
So that $Div(X) \simeq Div(\overline{\widetilde{X}})$,
but $F$ is omitted from the basis of $Div(X)$
and $\overline{E}_P$ is added as the basis of $Div(\overline{\widetilde{X}})$.

However, we need not replace the basis for $Pic$.
Let $\rho^\ast=\overline{\pi}_2 \circ \pi_1^\ast$ be
the $\mathbb{Z}$-linear map from $Div(X)$ to $Div(\overline{\widetilde{X}})$.
$\rho^\ast$ is written as
\begin{equation}
\rho^\ast(D)=\overline{\widetilde{D-\lambda F}}+
(\widetilde D \cdot E_P) \overline{E}_P.
\end{equation}
$\rho^\ast$ maps $Div(X)$ to $Div(\overline{\widetilde{X}})$ bijectively,
and maps the principal divisor group to the principal divisor group.
So, $\rho^\ast$ induces an isomorphism of $Pic(X)$ to $Pic(\overline{\widetilde X})$.
Since $\rho^\ast$ maps $F$ to $\overline{E}_P$,
the divisor class of $F$ is mapped to the divisor class of $\overline{E}_P$.
(More precisely,
for a divisor $D$ on $X$,
$D \equiv F$ on $X$ is equivalent with
$\rho^\ast(D) \equiv \overline{E}_P$).

The intersection form on $\overline{\widetilde{X}}$ is given as follows.
\begin{eqnarray}
\overline{E}_P \cdot \overline{E}_P=0,\ 
\overline{\widetilde{C}} \cdot \overline{E}_P=C \cdot F
\text{ for } C \not= F. \nonumber \\
\overline{\widetilde{C}}_1 \cdot \overline{\widetilde{C}}_2=
C_1 \cdot C_2+(C_1 \cdot F)(C_2 \cdot F)
-(C_1 \cdot F)(\widetilde C_2 \cdot E_P) \\
-(\widetilde C_1 \cdot E_P)(C_2 \cdot F). \nonumber
\end{eqnarray}
The canonical divisor of $\overline{\widetilde{X}}$ is given by
\begin{equation}
\Omega_{\overline{\widetilde{X}}}=
\rho^\ast(\Omega_X)+\overline{E}_P=
\overline{\widetilde{\Omega_X-\lambda F}}+
\big\{ (\widetilde{\Omega}_X \cdot E_P)+1\big\}
\overline{E}_P.
\end{equation}
Of cource we have
$\Omega_{\overline{\widetilde{X}}} \cdot
\Omega_{\overline{\widetilde{X}}} =
\Omega_X \cdot \Omega_X$
and $\overline{E}_P \cdot \Omega_{\overline{\widetilde X}}=-2$.

\section{$\mathfrak{G}$-invariant divisors.} \label{chB}

\subsection{Redunction to the even degree case.}

We shall return to the $k$-rationality problem of $k(x,y,z)$
defined by $z^2=ay^2+P(x)$,
with the conditions (1) $\sim$ (5) in the subsection \ref{main}.

Without loss of generality,
we can assume that $\deg P=r$ is even,
by the following reason.

Suppose that $\deg P=r$ is odd, and put $r=2s-1$.
Put $x^\prime=\frac{1}{x}, y^\prime=x^{\prime s}y,
z^\prime=x^{\prime s}z$,
then $z^2=ay^2+P(x)$ is re-written as
$z^{\prime 2}=ay^{\prime 2}+x^{\prime2s}P(\frac{1}{x^\prime}).$
When $P(x)=\sum_{i=0}^{2s-1} a_i x^i$ with $a_{2s-1} \not= 0$,
$P_1(x) := x^{\prime 2s}P(\frac{1}{x^\prime})=
\sum_{i=0}^{2s-1} a_i x^{\prime 2s-i}$
is a polynomial with the degree $2s$.
Since $k(x,y,z)=k(x^\prime,y^\prime,z^\prime)$,
the $k$-rationality problem of $k(x,y,z)$ is reduced to
that of $k(x^\prime, y^\prime, z^\prime)$
for the polynomial $P_1(x)$ of even degree.

Since the root of $P_1(x)$ are $0$ and $\{\frac{1}{c_i}\}_{1 \leq i \le r}$,
where $\{c_i\}$ are the roots of $P(x)$,
the conditions (1) $\sim$ (5) are satisfied for $P_1(x)$ also.
Here we assume that $0$ is not a root of $P(x)$.
(If $P(0)=0$, then $P(x)$ is divided by $x$,
so $P(x)$ is not irreducible, thus $k(x,y,z)$ is not $k$-rational
by the discussion in \S\ref{Manin}.)

\subsection{Another biregularization of $T$.} \label{biregT}

As discussed in the subsection \ref{biregularization},
the birational transformation $T$ of $\mathbb{P}^1 \times \mathbb{P}^1$
defined by $x \mapsto x, u \mapsto \frac{P(x)}{u}$
becomes biregular after $2r$ blowings-up in total,
once at each $P_i: x=c_i, u=0 (1 \leq i \leq r)$
and $r$-times at $P_{r+1}: x=u=\infty$.

In this subsection, after reaching $X_r$ in the subsection \ref{biregularization},
we shall proceed in another way.
Blow up $X_r$ at the point  $P_{r+1}$ (this is $X_{r+1}$),
and blow down it by $x=\infty$.
We denote the obtained surface with $Y_1$.
(For simplicity, we shall omit $\widetilde{}$ and $\overline{\ }$
for blowing-up and down.
The confusion is avoided by seeing $C$ is a divisor of which surface).
In $Div(Y_1)$,
$(x=\infty)$ disappears and is replaced by $E_{r+1}$.
The only fundamental point of $T$ is $P_{r+2}$,
and the only exceptional curve is $E_{r+1}$.

Blow up $Y_1$ at $P_{r+2}$ and blow down by $E_{r+1}$.
In the obtained surface $Y_2$,
$E_{r+1}$ disappears and is replaced by $E_{r+2}$.
The only fundamental point of $T$ is $P_{r+3}$,
and the only exceptional curve is $E_{r+2}$.

Repeat this process $\frac{r}{2}$ times.
On the surface $Y_{\frac{r}{2}}$,
$T$ becomes biregular,
and maps $E_{\frac{3r}{2}}$ biregularly to
$E_{\frac{3r}{2}}$,
as studied in the subsection \ref{biregularization}.
We shall denote the obtained $Y_{\frac{r}{2}}$ by $Y$.

The rank of $Pic(Y)$ is $r+2$ with the basis $E_i (1 \leq i \leq r)$
and $E_{\frac{3r}{2}}, (u=\infty)$.

Let $\rho^\ast$ be the $\mathbb{Z}$-linear map from $Div(\mathbb{P}^1 \times \mathbb{P}^1)$
to $Div(Y)$,
obtained by the iteration of $\pi^\ast$ and $\rho^\ast$ mentioned at
the end of the subsection \ref{intersection} and at the subsection \ref{updown}.
Then for an irreducible curve $C$ on $\mathbb{P}^1 \times \mathbb{P}^1$
other than $(x=\infty)$,
we have
\begin{equation}
\rho^\ast(C)=C+\sum_{i=1}^r(C \cdot E_i)E_i+
\sum_{i=1}^{r/2} (C \cdot E_{r+i}) E_{\frac{3r}{2}},
\end{equation}
where $(C \cdot E_{r+i})$ is the intersection form  in $X_{r+i}$.
Especially, $\rho^\ast(u=\infty)=(u=\infty)+\frac{r}{2}E_{\frac{3r}{2}}$.
On the other hand,
$\rho^\ast(x=\infty)=E_{\frac{3r}{2}}$,
as stated in the subsection \ref{updown}.

Let $C$ be an irreducible curve corresponding to
an irreducible polynomial of degree $n$
with respect to $x$ and $m$ with respect to $u$.

Since $C \equiv n(x=\infty)+m(u=\infty)$
in $\mathbb{P}^1 \times \mathbb{P}^1$,
we have in $Y$
\begin{eqnarray} \label{Cequiv}
C \equiv m(u=\infty)-\sum_{i=1}^r (C \cdot E_i)E_i+\nu E_{\frac{3r}{2}} \\
\text{where } \nu=n+\frac{mr}{2}-\sum_{i=1}^{r/2}(C \cdot E_{r+i}) \nonumber
\end{eqnarray}
Especially
\begin{eqnarray}
(x=c_i) \equiv -E_i+E_{\frac{3r}{2}} \\
(u=0) \equiv (u=\infty)-\sum_{i=1}^r E_i+\frac{r}{2}E_{\frac{3r}{2}} \nonumber
\end{eqnarray}

Thus, the action of $\mathfrak{G}$ on $Pic(Y)$ is represented by
the following matrices with $E_1, E_2, \dots, E_r, (u=\infty), E_{\frac{3r}{2}}$
as the basis in this order
\begin{eqnarray}
\text{For } \sigma \in N=Gal(l/k(\sqrt{a})), &
g_\sigma=
\begin{pmatrix}
A_\sigma & 0 \\
0 & I_{r+2} \\
\end{pmatrix} \\
\text{For } \sigma \in \mathfrak{G} \setminus N, &
g_\sigma=
\begin{pmatrix}
-A_\sigma & -1 & 0 \\
0 & 1 & 0 \\
1 & r/2 & 1
\end{pmatrix} \nonumber
\end{eqnarray}
This is identical with (\ref{even1}),
and we see that $Pic(Y)$ is just the direct summand $M_1$ of $Pic(X)$.

We shall determine $\mathfrak{G}$-invariant divisor classes:
$Pic(Y)^\mathfrak{G}=\bigcap_{\sigma} ker(g_\sigma-id)$.
Since $\sum_{\sigma} g_\sigma$ is zero except
the last row and the $(r+1)$-th column,
$Pic(Y)^\mathfrak{G}$ is of rank 2
with the basis $E_{\frac{3r}{2}}$ and
$(u=0)+(u=\infty)=2(u=\infty)-\sum_{i=1}^r E_i+\frac{r}{2}E_{\frac{3r}{2}}$.

On the other hand, the canonical divisor is calculated by
$\Omega=\rho^\ast\big(-2(x=\infty)-2(u=\infty)\big)
+\sum_{i=1}^r E_i+\frac{r}{2}E_{\frac{3r}{2}}$,
and we have
\begin{equation} \label{Omega}
\Omega=-2(u=\infty)+
\sum_{i=1}^r E_i-(\frac{r}{2}+2)E_{\frac{3r}{2}}
\end{equation}
Therefore, the basis of $Pic(Y)^\mathfrak{G}$ is chosen as $E_{\frac{3r}{2}}$
and $\Omega$.
Namely, a $\mathfrak{G}$-invariant divisor class becomes a
$\mathbb{Z}$-linear combination of $E_{\frac{3r}{2}}$ and $\Omega$,
with $E_{\frac{3r}{2}} \cdot E_{\frac{3r}{2}}=0, \ 
\Omega \cdot E_{\frac{3r}{2}}=-2, \ 
\Omega \cdot \Omega=8-r$.

A $\mathfrak{G}$-invariant irreducible curve $C$
belongs to a $\mathfrak{G}$-invariant divisor class,
so comparing (\ref{Cequiv}) with (\ref{Omega}),
we have $m=2(C \cdot E_i)$ for $1 \leq i \leq r$.
Thus $m$ should be even,
and $C \cdot E_i$ should be $\frac{m}{2}$ for any $i$.

Hereafter, we shall denote with $2m$ (instead of $m$)
the degree of a $\mathfrak{G}$-invariant curve with respect to $u$.
Then, \ref{Cequiv} is rewritten as
\begin{eqnarray} \label{Cequiv2}
C \equiv -m\Omega+\nu E_{\frac{3r}{2}} \\
\text{where } \nu=n+m(\frac{r}{2}-2)-\sum_{i=1}^{r/2}(C \cdot E_{r+i})
\nonumber
\end{eqnarray}
Note that the coefficient of $\Omega$ is always $-m$,
while the coefficient of $E_{\frac{3r}{2}}$ depends on the local
behavior of $C$ around the point $x=u=\infty$.

\subsection{Proof of Theorem \ref{B}} \label{pfB}

If $k(x,y,z)$ is $k$-rational,
we have $k(x,y,z)=k(t,s)$ for some $t$ and $s$.
Then $\overline{k}(x,u)=\overline{k}(t,s)$,
so that $Y$ is $k$-birational with $\mathbb{P}^1 \times \mathbb{P}^1$.
Here ``$k$-birational'' means that
there exists a birational mapping
$\mathbb{P}^1 \times \mathbb{P}^1 \rightarrow Y$
which commutes with the action of $G=Gal(k_{sep}/k)$,
where $G$ acts trivially on $t$ and $s$.

Let $\Phi$ be a $k$-birational mapping $\mathbb{P}^1 \times \mathbb{P}^1
\rightarrow Y$.
After finite steps of blowings-up of $\mathbb{P}^1 \times \mathbb{P}^1$ and $Y$
respectively, $\Phi$ is lifted to a biregular mapping
$Z \rightarrow Z^\prime$.

For $a,b \in k$,
the lines $t=a$ and $s=b$ are $\mathfrak{G}$-invariant in
$\mathbb{P}^1 \times \mathbb{P}^1$,
so that their images are also $\mathfrak{G}$-invariant in $Y$ or in $Z^\prime$.
Suppose that $t=a$ is not an exceptional curve of $\Phi$
and does not pass through fundamental point of $\Phi$,
then the values of intersection form
$(t=a) \cdot (t=a)=0, (t=a) \cdot \Omega=-2$
are kept invariant under the blowings-up,
so the image $C$ in $Z^\prime$ also satisfies
$C \cdot C=0$ and $C \cdot \Omega=-2$ in $Z^\prime$.
\hfill\break
$Z^\prime$ is obtained from $Y$ by successive blowings-up
$\{E_j^\prime\}$.
By each blowing-up, $C \cdot C$ is decreased by $(C \cdot E_j^\prime)^2$
and $\Omega \cdot C$ is increased by $C \cdot E_j^\prime$,
by (\ref{2_1}) and (\ref{widetildeC}).

Thus we have
\begin{eqnarray} \label{CC}
C \cdot C=\sum_j m_j^2 \\
\Omega \cdot C=-2-\sum_j m_j \text{ on } Y \nonumber \\
\text{where } m_j=C \cdot E_j^2 \nonumber
\end{eqnarray}
From (\ref{Cequiv2}) we have
\begin{eqnarray}
C \cdot C &=& 4m\nu+\omega m^2 \\
\Omega \cdot C &=& -2\nu-m\omega
\text{ where } \omega=\Omega \cdot \Omega \nonumber
\end{eqnarray}
Combining with (\ref{CC}), we get
\begin{eqnarray} \label{meq}
\sum_j m_j^2=4m\nu+\omega m^2 \\
\sum_j m_j=2\nu+m\omega-2 \nonumber
\end{eqnarray}
As shown later,
we have $0 \leq m_j \leq 2m$ for any $j$.
Thus $0 \leq \sum_j m_j^2 \leq 2m\sum_j m_j$,
so that
\begin{equation}
4m\nu+\omega m^2 \leq 2m(2\nu+m\omega-2)
\end{equation}
which yields $\omega m^2 \geq 4m$.
This is impossible if $\omega \leq 0$ and $m \not= 0$.
Since $\omega=8-r$,
if $r \geq 8$,
then any $\mathfrak{G}$-invariant curve other than $x=$const.
cannot become the image of $t=a$.

The same holds for $s=b$.
Since $t$ and $s$ are algebraically independent,
at least one of $t$ and $s$ depends on $u$
so that $m \geq 1$.
This contradicts with the result above.
\begin{rem}
In the proof above,
the $k$-birationality of $\Phi$ is not used.
So we see that by any birational mapping
$\mathbb{P}^1 \times \mathbb{P}^1 \rightarrow Y$,
$t=a$ can not be mapped to a $\mathfrak{G}$-invariant curve
other than $x=$const., whenever $r \geq 8$.
\end{rem}

We shall prove $m_j \leq 2m$ below.

Let $E^\prime$ be the blowing-up of a ponit $P^\prime$ on $Y$.
$P^\prime$ lies on the line $x=c$ for some $c$ or on the line $E_{\frac{3r}{2}}$.
We have $C \cdot (x=c)=2m$ for $c \not= c_i$,
and $C \cdot (x=c_i)=C \cdot E_i=m$.
As for $E_{\frac{3r}{2}}$,
since $(x=c) \equiv E_{\frac{3r}{2}}$ for $c \not= c_i$,
we have $C \cdot E_{\frac{3r}{2}}=2m$.

Since $C_1 \cdot C_2 \geq 0$ for mutually different irreducible curves $C_1, C_2$,
after the blowings-up we have
$ \widetilde C \cdot (\widetilde{x=c})=C \cdot (x=c)-
(\widetilde C \cdot E^\prime)
\big( (\widetilde{x=c}) \cdot E^\prime \big) \geq 0$.
Since $(\widetilde{x=c}) \cdot E^\prime=1$,
we get
\begin{equation} \label{leq}
\widetilde{C} \cdot E^\prime \leq C \cdot (x=c) \leq 2m.
\end{equation}
Successive blowing-up may occur at a point on the once blown-up $E^\prime$.
Let $E^{\prime\prime}$ be the blowing-up of a point $P^{\prime\prime}$
on $E^\prime$.
Then, by the same reason as above,
we have $\widetilde C \cdot E^{\prime\prime} \leq C \cdot E^\prime$.

Therefore, for successive blowings-up $\{E_j^\prime\}$,
(namely $E_1^\prime$ is the blowing-up at $P_1 \in Y$,
$E_2^\prime$ is the blowing-up at $P_2 \in E_1^\prime$,
$E_3^\prime$ is the blowing-up at $P_3 \in E_2^\prime$,
and so on)
$C \cdot E_j^\prime$ is monotone decreasing.
Since the first blowing-up $E^\prime$ satisfies (\ref{leq}),
any further successive blowing-up also satisfies
$C \cdot E_j^\prime \leq 2m$.

\section{Further reduction}

\subsection{Del Pezszo surface.}

A non-singular projective surface $X$ is called a del Pezzo surface
if it is rational (namely, birational with $\mathbb{P}^2$ or
$\mathbb{P}^1 \times \mathbb{P}^1$ over $\overline{k}$)
and the anti-canonical divisor is ample.
The latter condition means that $\Omega \cdot \Omega >0$
and $\Omega \cdot \Gamma <0$ for every irreducible curve $\Gamma$ on $X$.

The following is a fundamental theorem for a del Pezzo surface.

\begin{thm}
A del Pezzo surface is biregular with some successive blowings-up of $\mathbb{P}^2$.
\end{thm}

In general, if $X$ is rational,
then some blowings-up of $X$ is biregular with some blowings-up of
$\mathbb{P}^1 \times \mathbb{P}^1$ or $\mathbb{P}^2$.
See the subsection \ref{birational}.
However, for a del Pezzao surface,
a blowing-up of $X$ is not required,
and only $\mathbb{P}^2$ should be blown-up.

For the proof of the theorem,
see for instance, Manin \cite{Man86} \S24
or Nagata \cite{Nag60a,Nag60b}.

Conversely, $n$-point blowing-up of $\mathbb{P}^2$ is a del Pezzao surface
if $n \leq 8$ and any three points do not lie on a same line
(for $n \geq 3$) and any six points do not lie
on a same quadratic curve (for $n \geq 6$).

\subsection{Iteration of blowings-up and down.}

Let $X$ be a non-singular porjective surface,
and $P_1, P_2, \dots, P_r$ be points on $X$.
The successive blowings-up at $\{P_i\}_{1 \leq i \leq r}$
does not depend on the order of the blowings-up.
(More precisely, the obtained surface by the blowings-up
in different orders are mutually biregular).

For successive blowings-up on the once blowing-up $E$,
(namely $E_1$ is the blowing-up at $P_1 \in X$,
$E_2$ is the blowing-up at $P_2 \in E_1$,
$E_3$ is the blowing-pu at $P_3 \in E_2$,
and so on)
the order of the blowing-up can not be changed.
In this case $C \cdot E_i$ is monotone decreasing,
as remarked at the end of the subsection \ref{pfB}.

Let $X_1$ be the blowing-up of $X$ at a point $P_1$,
and $Y_1$ be the blowing-down by some $\widetilde{F}_1$,
where $F_1$ is an irreducible curve on $X$
passing through $P_1$
such that $F_1 \cdot F_1=0$ and $F_1 \cdot \Omega=-2$ on $X$.

The blowing-up of $Y_1$ at some point $Q_1$ of $Y_1$ is
biregular with $X_1$, mapping $E_Q$ to $\widetilde{F}_1$,
by the definition of the blowing-down.

Let $X_2$ be the blowing-up of $X_1$ at $P_2 \in X_1 \setminus \widetilde{F}_1$.
Since $X_1 \setminus \widetilde{F}_1$ is biregular
with $Y_1 \setminus \{Q_1\}$,
this induces a blowing-up of $Y_1$ at the corresponding point $P_2^\prime$.
Blow-down again by some $\widetilde{F}_2$,
and let $Y_2$ be the obtained surface.

The blowing-up of $Y_2$ at some point $Q_2$ is biregular
with the blowing-up of $Y_1$ at $P_2^\prime$.
Since $X_2$ is biregular with the successive blowings-up of $Y_1$
at $Q_1$ and $P_2^\prime$,
we see that $X_2$ is biregular with the successive blowings-up of $Y_2$
at $Q_2$ and $Q_1$.

Repeat this $r$-times.
Let $X_r$ be the surface obtained from $X$ by the successive blowings-up at $\{P_i\}$.
After each blowing-up, take a suitable blwoing-down,
and after repeating this $r$-times,
let $Y_r$ be the obtained surface.
Then $X_r$ is biregular with the successive blowings-up of $Y_r$ at $\{Q_i\}$.
Here we assume that $P_i$ does not lie on $F_j$ for $j<i$ in $X_{i-1}$.
(For simplicity, we omit $\widetilde{}$ and $\overline{\ }$
for blowing-up and down).

\section{The case $\deg P=4$ or $6$.}

\subsection{The surface $Y_C$} \label{Y_C}

We shall continue the discussion in \S\ref{chB}.

When $\deg P=4$ or $6$,
we assume the existence of a $\mathfrak{G}$-invariant irreducible curve $C$
on $Y$ other than $x=$const.
such that $C \cdot C=0$ and $\Omega \cdot C=-2$
after some blowings-up of $Y$,
and show a contradiction occurs.

Let $\{E_j^\prime\}$ be necessary blowings-up,
then $m_j=C \cdot E_j^\prime$
must satisfy (\ref{meq})
in the subsection \ref{pfB},
where $m$ and $\nu$ are given by (\ref{Cequiv2}) in the subsection \ref{biregT}

We have $0 \leq m_j \leq 2m$ for any $j$.
We can suppose that for the first $n^\prime$ blowings-up,
we have $m_j>m$,
while for the remaining blowings-up, we have $m_j \leq m$.
This is possible since the order of the blowings-up can be chosen arbitrary
when $\{P_j^\prime\}$ are points of the same surface,
and $\{m_j\}$ is monotone decreasing when $P_j^\prime$
is a point on the once blown-up $E_{j^\prime}^\prime (j^\prime<j)$.

Let $P_1^\prime$ be a point on $Y$ such that $m_1 > m$.
$P_1^\prime$ lies on $x=c$ for some $c$ or on $E_{\frac{3r}{2}}$.
Note that $c \not= c_i$,
since $P_i$ lies on $x=c_i$
and $C \cdot E_i=m$,
so that for any point $P$ on $x=c_i$,
$C \cdot E_p >m$ can not occur.

So $(x=c) \equiv E_{\frac{3r}{2}}$ in $Pic(Y)$,
thus we have $(x=c) \cdot (x=c)=0$,
$(x=c) \cdot \Omega=-2$
and $C \cdot (x=c)=2m$.

Let $Y_1$ be the blowing-up of $Y$ at $P_1^\prime$,
and $Y_1^\prime$ be the blowing-down of $Y_1$
by $x=c$ (or $E_{\frac{3r}{2}}$).
$Y_1$ is biregular with the blowings-up of $Y_1^\prime$
at some $Q_1$, where
\begin{equation}
C \cdot E_{Q_1} =2m-C \cdot E_1^\prime=2m-m_1
\end{equation}
(To avoid the confusion,
we shall use $\widetilde{}$ and $\overline{\ }$.
Then
$\widetilde{\overline{\widetilde C}} \cdot E_{Q_1}=
\widetilde C \cdot (\widetilde{x=c})=
C \cdot (x=c)-\widetilde{C} \cdot E_1^\prime=2m-m_1$)

Let $P_2^\prime$ be a point on $Y_1$ such that $m_2 >m$.
$P_2^\prime$ does not lie on $(\widetilde{x=c})$,
because $\widetilde C \cdot (\widetilde{x=c})=2m-m_1<m$.
so that any point $P$ on $(\widetilde{x=c})$
can not satisfy $\widetilde{C} \cdot E_P >m$.

Since $Y_1 \setminus (\widetilde{x=c})$ is biregular
with $Y_1^\prime \setminus \{Q_1\}$,
the blowing-up at $P_2^\prime$ induces the blowing-up of $Y_1^\prime$
at the correspoinding point $P_2^{\prime\prime}$.
After the blowing-up at $P_2^{\prime\prime}$,
blow-down by $x=c^\prime, E_{\frac{3r}{2}}$ or $E_1^\prime$
passing through $P_2^{\prime\prime}$.
Let $Y_2^\prime$ be the obtained surface.
Then the blowing-up of $Y_1$ at $P_2^\prime$ is biregular with the successive blowings-up
of $Y_2^\prime$ at $Q_2$ and $Q_1$,
and we have $C \cdot E_{Q_2}=2m-m_2$.

Repeat this procedure $n^\prime$ times.
We shall denote the obtained surface $Y_{n^\prime}^\prime$
with $Y_C$.
($Y_C$ depends on $C$).
Then the successive blowings-up of $Y$ at $\{P_j^\prime\}$
is obtained by the successive blowings-up of $Y_C$,
but this time every blowing-up
satisfies $m_j^\prime := C \cdot E_{Q_j} \leq m$.

\subsection{Proof of Theorem \ref{C}}

We shall prove that $Y_C$ is a del Pezzo surface.

Let $E_0$ be $E_{\frac{3r}{2}}$ or its replacement
when $E_{\frac{3r}{2}}$ is eliminat4ed by some blowing down.
The relation (\ref{Cequiv2}) in $Pic(Y)^\mathfrak{G}$ yields the following
relation in $Pic(Y_C)^\mathfrak{G}$.
\begin{eqnarray}
C \equiv -m\Omega +\nu^\prime E_0 \\
\text{where } \nu^\prime=\nu+\sum_{j=1}^{n^\prime}(m-m_j) \nonumber
\end{eqnarray}
The relation (\ref{meq}) yields
\begin{eqnarray}
\sum_j m_j^{\prime2}=4m\nu^\prime+\omega m^2 \nonumber\\
\sum_j m_j^\prime=2\nu^\prime+m\omega-2 \\
\text{where } m_j^\prime=Min(m_j,2m-m_j) \nonumber
\end{eqnarray}
This time,
since $0 \leq m_j^\prime \leq m$,
we must have $0 \leq \sum_j m_j^{\prime2} \leq m \sum_j m_j^\prime$
so that
\begin{equation}
4m\nu^\prime+\omega m^2 \leq m(2\nu^\prime+m\omega-2)
\end{equation}
which is equivalent with $\nu^\prime \leq -1$.

Since $\omega:=\Omega \cdot \Omega=8-r>0$,
in order to prove that $Y_C$ is a del Pezzo surface,
it suffices to show $\Omega \cdot \Gamma<0$
for any irreducible curve $\Gamma$.

For $\Gamma=C$,
we have $\Omega \cdot C=-2-\sum_j m_j^\prime \leq -2$.
For $\Gamma=E_i$,
we have $\Omega \cdot E_i=-1$.
For $\Gamma=E_0$,
we have $\Omega \cdot E_0=-2$.
For $\Gamma$ other than mentioned above,
since $\Gamma \cdot C \geq 0$, we have
$\Gamma \cdot (-m\Omega+\nu^\prime E_0) \geq 0$,
so that $m\Omega \cdot \Gamma \leq \nu^\prime E_0 \cdot \Gamma$.
Since $\nu^\prime \leq -1$ and $E_0 \cdot \Gamma \geq 0$,
we have $\Omega \cdot \Gamma \leq 0$,
and $\Omega \cdot \Gamma$ can be 0
only if $E_0 \cdot \Gamma=0$.

Suppose that $E_0 \cdot \Gamma=0$.
Since $E_0$ is $\mathfrak{G}$-invariant,
for any image $\Gamma^\sigma$ by the action of $\sigma \in \mathfrak{G}$,
we have $E_0 \cdot \Gamma^\sigma=0$,
so that $E_0 \cdot \sum_{\sigma} \Gamma^\sigma=0$.
$\sum_{\sigma} \Gamma^\sigma$ is
$\mathfrak{G}$-invariant,
so can be written as $m_\Gamma \Omega+\nu_\Gamma E_0$
in $Pic(Y_C)^\mathfrak{G}$.
Since $E_0 \cdot E_0=0$ and $\Omega \cdot E_0=-2$,
we have $m_\Gamma=0$.
This implies that $\Gamma$ does not depend on $u$,
so $\Gamma=(x=c)$ for some $c$.

However, $(x=c) \equiv E_0$ for $c \not= c_i$,
and $(x=c_i)  \equiv -E_i+E_0$,
so we have $\Omega \cdot (x=c)=\Omega \cdot E_0=-2$
and $\Omega \cdot (x=c_i)=\Omega \cdot E_0-\Omega \cdot E_i
=-2-(-1)=-1$.

Now, the proof is completed,
and $Y_C$ is certainly a del Pezzo suface.

\begin{rem}
We have $C \cdot C=4m\nu^\prime+\omega m^2 >0$,
hence $-\frac{\omega}{4} m < \nu^\prime \leq -1$.

The possibility of $C \cdot C=0$
is discarded by the following reason.
Suppose that $C$ corresponds to $t=a$ and $C^\prime$ corresponds to $s=b$,
then $C \cdot C^\prime=1$ on $Z^\prime$.
$C \cdot C=0$ on $Y_C$ implies $m_j^\prime=0$
for any $j$,
so that $C \cdot C^\prime=1$ on $Y_C$,
which contradicts to $E_0 \cdot E_0=0$
and $E_0 \cdot \Omega=-2$.
\end{rem}

\subsection{Desired contradiction.}

Since $Y_C$ is a del Pezzo surface,
it is biregular with a five point (when $r=4$)
or seven point (when $r=6$) blowing-up $B$ of $\mathbb{P}^2$.
Suppose that $B$ is obtained from $\mathbb{P}^2$
by the blowings-up $F_i$ at $Q_i \in \mathbb{P}^2$.
Any three point of $Q_i$ do not lie on a same line,
and when $r=6$, any six point of $Q_i$
do not lie on a same quadratic curve.

Let $l$ be a line on $\mathbb{P}^2$
which does not pass through any $Q_i$.
Then $Pic(B)$ is a free $\mathbb{Z}$-module
with $F_i (1 \leq i \leq r+1)$ and $l$ as the basis.
The intersection form in $Pic(B)$ is written as
$\begin{pmatrix} -I_{r+1} & 0 \\ 0 & 1 \end{pmatrix}$
with the above basis in this order.
The canonical divisor is
$\Omega=-3l+\sum_{i=1}^{r+1}F_i$.

The biregular mapping of $Y_C$ and $B$ induces an isomorphism
of $Pic(Y_C)$ and $Pic(B)$,
keeping the intersection forms invariant.

On $Y_C$, the family of lines $x=c$ belongs to the divisor class $E_0$,
which satisfies $E_0 \cdot E_0=0$
and $E_0 \cdot \Omega=-2$.
This family should be mapped to a family of curves on $B$
whose divisor class $\Gamma$ satisfies
$\Gamma \cdot \Gamma=0$ and $\Omega \cdot \Gamma=-2$.

Now, we shall consider the case of $r=4$.
$B$ is a five point blowing-up of $\mathbb{P}^2$.
There are ten families of curves which satisfy
$\Gamma \cdot \Gamma=0$ and $\Omega \cdot \Gamma=-2$.
They are the families of lines passing through some $Q_i$,
say $Q_1$ whose divisor class is $l-F_1$,
and the families of quadratic curves passing through four points of $Q_i$,
whose divisor class is $2l-F_1-F_2-F_3-F_4$.

The family $x=c$ should be mapped to one of these ten families.
Let $\Pi$ be its divisor class.
Which class $\Pi$ may be,
we observe that there are another family whose divisor calss $\Pi_1$ satisfies
$\Pi+\Pi_1=-\Omega$.
This means that in $Pic(Y_C)$,
the divisor class $-\Omega-E_0$ is the class of some 
family of irreducible curves.
These curves are mutually disjoint and the union covers all $Y_C$,
since the curves of $\Pi_1$ behave so on $B$.

Let $E_0^\prime=-\Omega-E_0$
and change the basis of $Pic(Y_C)^\mathfrak{G}$
to $E_0^\prime$ and $\Omega$.
Then $C \equiv -\nu^\prime E_0^\prime-(m+\nu^\prime)\Omega$.

Note that $-m < \nu^\prime \leq -1$
implies $-\nu^\prime>0$
and $0 < m+\nu^\prime<m$.
After some blowings-up of $Y_C$
we reach to $Z^\prime$,
and some $m_j^\prime$ may be larger than $m+\nu^\prime$.
For such blowings-up, we blow-down by the curve belonging to $E_0^\prime$,
and obtain the surface $Y_C^\prime$.
(The procedure in the subsection \ref{Y_C},
using $E_0^\prime$ instead of $E_0$).

$Y_C^\prime$ is another del Pezzao surface and we have
$C \equiv \nu^{\prime\prime}E_0^\prime-(m+\nu^\prime)\Omega$
in $Pic(Y_C^\prime)$ for some
$-(m+\nu^\prime) < \nu^{\prime\prime} \leq -1$.
Again $Y_C^\prime$ is biregular with
some five point blowing-up of $\mathbb{
P}^2$,
and $-\Omega-E_0^\prime$ is the divisor class of a family
of irreducible curves on $Y_C^\prime$.

Repeat this procedure,
then after finite steps, the condition $m+\nu^\prime+\nu^{\prime\prime}
+\cdots+\nu^{(n)}>0$,
which means $C \cdot C>0$ in $Y_C^{(n)}$,
can not be satisfied.
This is a contradiction and the proof of $k$-irrationality of $k(x,y,z)$ is completed.

Next consider the case of $r=6$.
$B$ is a seven point blow up of $\mathbb{P}^2$,
and the families of curves whose divisor class $\Gamma$ satisfy
$\Gamma \cdot \Gamma=0$ and $\Omega \cdot \Gamma=-2$
are as follows:
\begin{sub}
\item[(1)]
$l-F_1$, etc (7 families)
\item[(2)]
$2l-F_1-F_2-F_3-F_4$, etc (35 families)
\item[(3)]
$3l-2F_1-F_2-F_3-F_4-F_5-F_6$ (42 families)
\item[(4)]
$4l-2F_1-2F_2-2F_3-F_4-F_5-F_6-F_7$ (35 families)
\item[(5)]
$5l-2F_1-2F_2-2F_3-2F_4-2F_5-2F_6-F_7$ (7 families)
\end{sub}

The families of lines $x=c$ on $Y_C$ should be mapped to one of these 126 families on $B$.
Let $\Pi$ be its divisor class.
We observe that,
which class $\Pi$ may be,
there are another family $\Pi_1$
such that $\Pi+\Pi_1=-2\Omega$.
In other words,
$-E_0-2\Omega$ is the divisor class of some family of irreducible curves.

Put $E_0^\prime=-E_0-2\Omega$
and change the basis of $Pic(Y_C)^\mathfrak{G}$
to $E_0^\prime$ and $\Omega$.
Then $C \equiv -\nu^\prime E_0^\prime-(m+2\nu^\prime)\Omega$
in $Pic(Y_C)$.
After the similar blowings-up and down to those
in the proof for $r=4$,
we get $C \equiv \nu^{\prime\prime}E_0^\prime
-(m+2\nu^\prime)\Omega$ in $Y_C^\prime$
for some $\nu^{\prime\prime}$ such that
$-(\frac{m}{2}+\nu^\prime) < \nu^{\prime\prime} \leq -1$.

Repeat this procedure,
and after finiter steps,
the condition $m+2\nu^\prime+2\nu^{\prime\prime}+\cdots
+2\nu^{(n)}>0$,
which means $C \cdot C>0$ in $Y_C^{(n)}$,
can not be satisfied.
This is a contradiction and the $k$-irrationality
of $k(x,y,z)$ has been proved.

\section*{Acknowledgement}
The author would like to thank Professor 
Ming-chang Kang who gave many valuable advices and comments.


\begin{thebibliography}{BCSS85}

\bibitem[BCSS85]{BCSS85}
A. Beauville, J.-L. Colliot-Th\'el\`ene, J.-J. Sansuc, P. Swinnerton-Dyer,
{\em Vari\'et\'es stablement rationnelles non rationnelles}
Annals of Math. {\bf 121} (1985), no. 2, 283--318.

\bibitem[Ch\^a59]{Cha59}
F. Ch\^atelet,
{\em Points rationnels sur certaines courbes et surfaces cubiques},
Enseignement Math. (2) {\bf 5} (1959) 153--170.

\bibitem[Har77]{Har77}
R. Hartshorne
{\em Algebraic Goemetry} Especially Chap. 5,
Grad. Texts in Math. Vol.52, Springer-Verlag, (1977).

\bibitem[Isk67]{Isk67}
V. A. Iskovskih,
{\em Rational surfaces with a pencil of rational curves},
Math. USSR Sb. {\bf 3} (1967), 563--587.

\bibitem[Isk96]{Isk96}
V. A. Iskovskih,
{\em On a rationality criterion for conic bundles},
Sb. Math. {\bf 187} (1996), 1021--1038.

\bibitem[Isk97]{Isk97}
V. A. Iskovskih,
{\em On the rationality problem for algebraic threefolds},
Proc. Steklov Inst. Math. {\bf 218} (1997), 186--227.

\bibitem[Man67]{Man67}
Yu. I. Manin,
{\em Rational surfaces over perfect fields. II},
Math. USSR Sb. {\bf 1} (1967), 141--168.

\bibitem[Man86]{Man86}
Yu. I. Manin,
{\em Cubic forms. Algebra, geometry, arithmetic} Especially \S 24,
Translated from the Russian by M. Hazewinkel. Second edition,
North-Holland Publishing Co., Amsterdam, (1986).

\bibitem[Nag60a]{Nag60a}
M. Nagata
{\em On rational surfaces. I.
Irreducible curves of arithmetic genus 0 or 1},
Mem. Coll. Sci. Univ. Kyoto Ser. A Math. {\bf 32}
(1960) 351--370.

\bibitem[Nag60b]{Nag60b}
M. Nagata
{\em On rational surfaces. II},
Coll. Sci. Univ. Kyoto Ser. A Math. {\bf 33}
(1960) 273--293. 

\end{thebibliography}
\end{document}